\documentclass[reqno,12pt]{amsart}
\usepackage{amsmath,amsthm,amssymb,amsfonts,amscd}
\usepackage{epsfig}
\usepackage{xypic}
\xyoption{frame}
\numberwithin{equation}{section}

\theoremstyle{plain}
\newtheorem{theorem}{Theorem}

\newtheorem{corollary}[theorem]{Corollary}
\newtheorem{proposition}[theorem]{Proposition}
\newtheorem*{theorem*}{Theorem}
\newtheorem*{conjecture*}{Conjecture}

\theoremstyle{definition}

\newtheorem{example}[theorem]{Example}
\newtheorem*{definition}{Definition}

\newcommand{\CC}{{\mathbb{C}}}

\newcommand{\QQ}{{\mathbb{Q}}}

\newcommand{\RR}{{\mathbb{R}}}
\newcommand{\ZZ}{{\mathbb{Z}}}

\newcommand{\calC}{{\mathcal C}}

\newcommand{\ff}{{\bf f}}
\newcommand{\hh}{{\bf h}}

\begin{document}
\title[Complete intersection singularities]{Strange duality between hypersurface and complete intersection singularities}
\author{Wolfgang Ebeling and Atsushi Takahashi}
\thanks{Partially supported 
by the DFG-programme SPP1388 ''Representation Theory'' and 
by JSPS KAKENHI Grant Number 24684005.}
\address{Institut f\"ur Algebraische Geometrie, Leibniz Universit\"at Hannover, Postfach 6009, D-30060 Hannover, Germany}
\email{ebeling@math.uni-hannover.de}
\address{
Department of Mathematics, Graduate School of Science, Osaka University, 
Toyonaka Osaka, 560-0043, Japan}
\email{takahashi@math.sci.osaka-u.ac.jp}
\date{}

\begin{abstract} C.T.C.~Wall and the first author discovered an extension of Arnold's strange duality embracing on one hand series of bimodal hypersurface singularities and on the other, isolated complete intersection singularities. In this paper, we derive this duality from the mirror symmetry and the Berglund-H\"ubsch transposition of invertible polynomials.
\end{abstract}
\maketitle
\section*{Introduction}
During his classification of  hypersurface singularities, Arnold \cite{A} observed a strange duality between the 14 exceptional unimodal singularities. C.T.C.~Wall and the first author \cite{EW} discovered an extension of this duality embracing on one hand series of bimodal singularities and on the other, isolated complete intersection singularities (ICIS) in $\CC^4$. The duals of the complete intersection singularities are not themselves singularities, but are virtual ($k=-1$) cases of series (e.g.\  $W_{1,k} : k \geq 0$) of bimodal singularities. They associated to these well-defined Coxeter-Dynkin diagrams and Milnor lattices and showed that all numerical features of Arnold's strange duality continue to hold.

The objective of this paper is to derive this extended strange duality from the mirror symmetry and the Berglund-H\"ubsch transposition of invertible polynomials. The bimodal series start with singularities with $k=0$ (e.g.\ $W_{1,0}$). They can be given by polynomials with two moduli. Setting one of the moduli equal to zero, one is left with a one-parameter family of weighted homogeneous polynomials. 
It is natural from the mirror symmetry view point to expect that adding one monomial to an invertible polynomial is dual to 
having another $\CC^\ast$-action on the dual polynomial, which leads to our duality between virtual singularities and complete intersection singularities.

We shall proceed as follows. We first classify the non-degenerate invertible polynomials with a $\ZZ/2\ZZ$-action. They are defined by certain $3 \times 3$-matrices. Then we shall classify the possibilities to extend such a $3 \times 3$-matrix to a certain $4 \times 3$-matrix satisfying certain conditions. Such a matrix defines a polynomial $\ff(x,y,z)$ with four monomials with a non-isolated singularity. We shall consider the Berglund-H\"ubsch transpose of this polynomial. The kernel of the transpose $3 \times 4$-matrix defines a $\CC^\ast$-action on the space $\CC^4$ and this matrix and the degree 0 polynomials define a complete intersection singularity in $\CC^4$ as the zero set of two polynomials. 

Following \cite{ET1}, we consider the polynomial $\ff(x,y,z)-xyz$. Under certain conditions, there is a coordinate transformation which transforms this polynomial to a polynomial $\hh(x,y,z)-xyz$, where $\hh$ is again a polynomial with four monomials, but now has an isolated singularity at the origin. We call this a {\em virtual singularity}. The polynomial $\hh$ is no longer weighted homogenous but its Newton polygon at infinity has two two-dimensional faces. We thus obtain a duality between the virtual hypersurface singularities and complete intersection singularities.

We show that this duality has the features of Arnold's strange duality. Namely, we associate Dolgachev and Gabrielov numbers to the polynomials $\hh$ and the equations defining the complete intersection singularities generalizing the approach of \cite{ET1}. It turns out that the Dolgachev numbers of the polynomial $\hh$ are the Gabrielov numbers of the pair of polynomials defining the complete intersection singularity and vice versa, the Gabrielov numbers of the polynomial $\hh$ are the Dolgachev numbers of the pair of polynomials defining the complete intersection singularity. Moreover, we show that the reduced zeta function of the monodromy at infinity of a virtual singularity coincides with the product of the Poincar\'e series of the coordinate ring of the dual complete intersection singularity and a polynomial encoding its Dolgachev numbers.

As an example we consider those singularities with Gorenstein parameter being equal to 1. In this way, we recover precisely the virtual singularities of the bimodal series and the extension of Arnold' strange duality found in \cite{EW}. Finally, we construct Coxeter-Dynkin diagrams for the virtual bimodal singularities and show that they can be transformed to graphs which have the same shape as in the exceptional unimodal case used in Gabrielov's original definition of the numbers now named after him.

\section{Invertible polynomials}
We recall some general definitions about invertible polynomials.

Let $f(x_1,\dots, x_n)$ be a  weighted homogeneous polynomial, namely, a polynomial with the property that there are positive integers $w_1,\dots ,w_n$ and $d$ such that 
$f(\lambda^{w_1} x_1, \dots, \lambda^{w_n} x_n) = \lambda^d f(x_1,\dots ,x_n)$ 
for $\lambda \in \CC^\ast$. We call $(w_1,\dots ,w_n;d)$ a system of {\em weights}. 
\begin{definition}
A  weighted homogeneous polynomial $f(x_1,\dots ,x_n)$ is called {\em invertible} if 
the following conditions are satisfied:
\begin{enumerate}
\item the number of variables ($=n$) coincides with the number of monomials 
in the polynomial $f(x_1,\dots x_n)$, 
namely, 
\[
f(x_1,\dots ,x_n)=\sum_{i=1}^na_i\prod_{j=1}^nx_j^{E_{ij}}
\]
for some coefficients $a_i\in\CC^\ast$ and non-negative integers 
$E_{ij}$ for $i,j=1,\dots, n$,
\item a system of weights $(w_1,\dots ,w_n;d)$ can be uniquely determined by 
the polynomial $f(x_1,\dots ,x_n)$ up to a constant factor ${\rm gcd}(w_1,\dots ,w_n;d)$, 
namely, the matrix $E:=(E_{ij})$ is invertible over $\QQ$.
\end{enumerate}
An invertible polynomial is called {\em non-degenerate}, if it has an isolated singularity at the origin. 
\end{definition}

Without loss of generality we shall assume that $\det E > 0$.

An invertible polynomial has a {\em canonical system of weights} $W_f=(w_1, \ldots , w_n;d)$ given by the unique solution of the equation
\begin{equation*}
E
\begin{pmatrix}
w_1\\
\vdots\\
w_n
\end{pmatrix}
={\rm det}(E)
\begin{pmatrix}
1\\
\vdots\\
1
\end{pmatrix}
,\quad 
d:={\rm det}(E).
\end{equation*}
This system of weights is in general non-reduced, i.e.\  in general $c_f:= {\rm gcd}(w_1, \ldots , w_n,d)>1$.

\begin{definition}
Let $f(x_1,\dots ,x_n)=\sum_{i=1}^na_i\prod_{j=1}^nx_j^{E_{ij}}$ be an invertible polynomial. Consider the free abelian group $\oplus_{i=1}^n\ZZ\vec{x_i}\oplus \ZZ\vec{f}$ 
generated by the symbols $\vec{x_i}$ for the variables $x_i$ for $i=1,\dots, n$
and the symbol $\vec{f}$ for the polynomial $f$.
The {\em maximal grading} $L_f$ of the invertible polynomial $f$ 
is the abelian group defined by the quotient 
\[
L_f:=\bigoplus_{i=1}^n\ZZ\vec{x_i}\oplus \ZZ\vec{f}\left/I_f\right.,
\]
where $I_f$ is the subgroup generated by the elements 
\[
\vec{f}-\sum_{j=1}^nE_{ij}\vec{x_j},\quad i=1,\dots ,n.
\]
\end{definition}

\begin{definition}
Let $f(x_1,\dots ,x_n)$ be an invertible polynomial and $L_f$ be the maximal grading of $f$.
The {\em maximal abelian symmetry group} $\widehat{G}_f$ of $f$ is the abelian group defined by 
\[
\widehat{G}_f:={\rm Spec}(\CC L_f),
\]
where $\CC L_f$ denotes the group ring of $L_f$. Equivalently, 
\[
\widehat{G}_f=\left\{(\lambda_1,\dots ,\lambda_n)\in(\CC^\ast)^n \, \left| \,
\prod_{j=1}^n \lambda_j ^{E_{1j}}=\dots =\prod_{j=1}^n\lambda_j^{E_{nj}}\right\} \right..
\]
Moreover, we define
\[
G_f=\left\{(\lambda_1,\dots ,\lambda_n)\in \widehat{G}_f \, \left| \, \prod_{j=1}^n \lambda_j ^{E_{1j}}=\dots =\prod_{j=1}^n\lambda_j^{E_{nj}}=1 \right\} \right..
\]
\end{definition}

Let $f(x_1,\dots ,x_n)$ be an invertible polynomial and $W_f=(w_1, \ldots , w_n;d)$ be the canonical system of weights associated to $f$. Set 
\[ q_i := \frac{w_i}{d}, \quad i=1, \ldots , n.
\]
Note that $G_f$ always contains the {\em exponential grading operator}
\[
g_0:=(\exp(2\pi \sqrt{-1}q_1), \ldots , \exp(2\pi \sqrt{-1}q_n)).
\]
Let $G_0$ be the subgroup of $G_f$ generated by $g_0$. One has (cf.\ \cite{ET})
\[
[G_f : G_0]=c_f.
\]
Let $f(x_1,\dots ,x_n)=\sum_{i=1}^na_i\prod_{j=1}^nx_j^{E_{ij}}$ be an invertible polynomial. Following \cite{BH}, the {\em Berglund-H\"ubsch transpose} of $\widetilde{f}(x_1, \ldots , x_n)$ of $f$ is defined by 
\[ 
\widetilde{f}(x_1,\dots ,x_n)=\sum_{i=1}^na_i\prod_{j=1}^nx_j^{E_{ji}}.
\]
By \cite{BHe}, for a subgroup $G \subset G_f$ its {\em dual group}  $\widetilde{G}$ is defined by 
\[
\widetilde{G}:= {\rm Hom}(G_f/G, \CC^\ast).
\]
Note that ${\rm Hom}(G_f/G, \CC^\ast)$ is isomorphic to $G_{\widetilde{f}}$, see \cite{BHe}.
By \cite{Kr}, we have
\[
\widetilde{G}_0 = {\rm SL}_n(\ZZ) \cap G_{\widetilde{f}}. 
\]
Moreover, by \cite[Proposition 3.1]{ET}, we have $|\widetilde{G}_0|=c_f$.

\section{Invertible polynomials with a $\ZZ/2\ZZ$-action}
Let $f(x,y,z)$ be a non-degenerate invertible polynomial with $[G_f : G_0]=2$. We shall now classify the non-degenerate invertible polynomials with such a group action.

\begin{proposition}
There are the following non-degenerate invertible polynomials $f(x,y,z)$ with $[G_f : G_0]=2$.  We list the possible types and the conditions. The coordinates are chosen so that the action of $\widetilde{G}_0=\ZZ/2\ZZ$ on $\widetilde{f}$ is given by $(x,y,z) \mapsto (-x,-y,z)$.
\begin{itemize}
\item[{\rm I:}] $f(x,y,z)=x^{p_1}+y^{p_2}+z^{p_3}$; $p_1, p_2$ even, 
\item[{\rm IIA:}] $f(x,y,z)=x^{p_2}+xy^{p_3/p_2}+z^{p_1}$; $p_2$ odd, $p_3/p_2$ even, 
\item[{\rm IIB:}] $f(x,y,z)=x^{p_1}+y^{p_2}+yz^{p_3/p_2}$; $p_1, p_2$ even, 
\item[{\rm III:}] $f(x,y,z)=x^{q_2+1}y+xy^{q_3+1}+z^{p_1}$; $q_2, q_3$ even, 
\item[{\rm IV:}] $f(x,y,z)=x^{p_1}+xy^{\frac{p_2}{p_1}}+yz^{\frac{p_3}{p_2}}$; $p_2/p_1$ even,  $p_1$ odd.
\end{itemize}
\end{proposition}

\begin{proof} This follows by inspection of \cite[Table~1]{ET1}.
\end{proof}

Let $\widehat{G}_0$ be the subgroup of $\widehat{G}_f$ 
defined by the commutative diagram of short exact sequences
\[ 
\xymatrix{
\{ 1 \}\ar[r] & G_0 \ar[r]\ar@{^{(}->}[d]  & \widehat{G}_0 \ar[r]\ar@{^{(}->}[d] &  \CC^\ast \ar[r]\ar@{=}[d]& \{ 1 \}\\
\{ 1 \}\ar[r] & G_{f} \ar[r] & \widehat{G}_{f} \ar[r] & \CC^\ast \ar[r]& \{ 1 \}
}
\]
Let $L_0$ be the quotient of $L_{f}$ corresponding to the subgroup $\widehat{G}_0$ of $\widehat{G}_{f}$. 

We shall now classify $4 \times 3$-matrices $E = (E_{ij})^{i=1,2,3,4}_{j=1,2,3}$ such that 
\[ \ZZ \vec{x} \oplus \ZZ \vec{y} \oplus \ZZ \vec{z} \oplus \ZZ \vec{f} / \langle E_{i1} \vec{x} +E_{i2} \vec{y}+E_{i3} \vec{z}= \vec{f}, i=1, \ldots ,4 \rangle \cong L_0
\]
and $\calC_{(F,G_0)} := [ (F^{-1}(0) \setminus \{ 0 \})/ \widehat{G}_0]$, where $F:= \sum_{i=1}^4 a_i x^{E_{i1}}y^{E_{i2}}z^{E_{i3}}$, is a smooth projective line with 4 isotropic points whose orders are $\alpha_1, \alpha_2, \alpha_3, \alpha_4$, where $A_{(f,G_0)}=(\alpha_1, \alpha_2, \alpha_3, \alpha_4)$ are the Dolgachev numbers of the pair $(f,G_0)$ defined in \cite{ET}, for general $a_1, a_2, a_3,  a_4$.

\begin{proposition} The possible matrices $E$ are classified into the following types up to a permutation of the rows.  The matrices are described by the corresponding polynomials $F$.
\begin{itemize}
\item[{\rm I:}]  $a_1x^{p_1}+a_2y^{p_2}+a_3z^{p_3}+a_4x^{\frac{p_1}{2}}y^{\frac{p_2}{2}}$
\item[{\rm IIA:}] $a_1x^{p_2}+a_2xy^{\frac{p_3}{p_2}}+a_3z^{p_1}  +a_4x^{\frac{p_2+1}{2}}y^{\frac{p_3}{2p_2}}$
\item[{\rm IIB:}] $a_1x^{p_1} + a_2y^{p_2} + a_3yz^{\frac{p_3}{p_2}} +a_4x^{\frac{p_1}{2}}y^{\frac{p_2}{2}}$
\item[${\rm IIB}^\sharp:$] $(p_2=2)$ $a_1x^{\frac{p_1}{2}}z^{\frac{p_3}{2}}+a_2y^2 + a_3yz^{\frac{p_3}{2}} +a_4x^{\frac{p_1}{2}}y$
\item[{\rm III:}] $a_1x^{q_2+1}y + a_2xy^{q_3+1}+a_3z^{p_1} +a_4x^{\frac{q_2}{2}+1}y^{\frac{q_3}{2}+1}$
\item[{\rm IV:}] $a_1x^{p_1} + a_2xy^{\frac{p_2}{p_1}} +a_3yz^{\frac{p_3}{p_2}} +a_4x^{\frac{p_1+1}{2}}y^{\frac{p_2}{2p_1}}$
\item[${\rm IV}^\sharp:$] $(\frac{p_2}{p_1}=2)$ $a_1x^{\frac{p_1-1}{2}}z^{\frac{p_3}{p_2}}+a_2xy^2+a_3yz^{\frac{p_3}{p_2}}+a_4x^{\frac{p_1+1}{2}}y$
\end{itemize}
\end{proposition}

\begin{proof} We only give the proof for the case IIB which is the most difficult one. The other cases are easier and are treated analogously. 

In the case IIB, the group $L_f$ is the quotient of the abelian group $\ZZ \vec{x} \oplus \ZZ \vec{y} \oplus \ZZ \vec{z} \oplus \ZZ \vec{f} $ given by the relations $p_1 \vec{x}= p_2 \vec{y}= \vec{y} + \frac{p_3}{p_2} \vec{z} = \vec{f}$ and $L_0$ is given by the additional relation 
\begin{equation}
\frac{p_1}{2}\vec{x} = \frac{p_2}{2} \vec {y} \label{eq:xy}
\end{equation}
We derive from these relations the relation 
\begin{equation}
(p_2-1) \vec{y} = \frac{p_3}{p_2} \vec{z}. \label{eq:yz}
\end{equation}

(1) We first classify all monomials which can appear in $F$. Suppose a monomial $x^ay^b$ appears. 
Then we must have
\[ a\vec{x} + b \vec{y} = \vec{f} = \frac{p_1}{2} \vec{x} + \frac{p_2}{2} \vec{y}. 
\]
From this we get
\[ \left( \frac{p_1}{2} -a \right) \vec{x} = \left( b - \frac{p_2}{2} \right) \vec{y}. \]
By Relation (\ref{eq:xy}) there must exist an integer $c$ such that
\[ \frac{p_1}{2} -a = \frac{p_1}{2} c \quad \mbox{and} \quad b -  \frac{p_2}{2} =  \frac{p_2}{2} c. \]
But
\[  \frac{p_1}{2}(1-c)=a \geq 0  \quad \mbox{and} \quad \frac{p_2}{2} (c+1) = b \geq 0. \]
This implies $c=-1,0,1$. Therefore we obtain the possibilities
\[ x^{\frac{p_1}{2}}y^{\frac{p_2}{2}}, y^{p_2} \mbox{ or } x^{p_1}. \]
In a similar way, using the relation (\ref{eq:yz}),  we can derive the possibilities
\[ y^{p_2}, yz^{\frac{p_3}{p_2}} \mbox{ or } z^{p_3}\ ( \mbox{if } p_2=2) \]
Now suppose that a monomial $x^az^b$ appears. Then 
\[ a \vec{x} + b \vec{z} = \vec{f} = \vec{y} + \frac{p_3}{p_2} \vec{z}. \]
From this it follows that $\vec{y} = c \vec{x}$ for some positive integer $c$ since $\vec{y} + \frac{p_3}{p_2} \vec{z}= \vec{f}$ is the only relation involving $\vec{z}$. Relation (\ref{eq:xy}) implies $p_2=2$, $\vec{y} = \frac{p_1}{2} \vec{x}$ and
\[ a\vec{x} + b \vec{z} = \frac{p_1}{2} \vec{x} + \frac{p_3}{2} \vec{z}. \]
This yields the possibilities
\[ x^{p_1}, x^{\frac{p_1}{2}}z^{\frac{p_3}{2}}\ ( \mbox{if } p_2=2) \mbox{ or } z^{p_3}\ ( \mbox{if } p_2=2) . \]
Finally one can derive that there are no monomials of the form $x^ay^bz^c$ with $a,b,c >0$. 

(2) Therefore, if $p_2 \neq 2$, we only obtain the possibility
\[ F(x,y,z)=a_1x^{p_1} + a_2y^{p_2} + a_3yz^{\frac{p_3}{p_2}} +a_4x^{\frac{p_1}{2}}y^{\frac{p_2}{2}}. \]
If $p_2=2$ we obtain several possibilities. In this case we have to consider the system of weights for $G_0$
\[ \left( \frac{p_3}{2}, \frac{p_1p_3}{4}, \frac{p_1}{2}; \frac{p_1p_3}{2} \right) \]
and the Dolgachev numbers of the pair $(f,G_0)$ given by \cite{ET}
\[ A_{(f,G_0)} = \left( \frac{p_1}{2}, \frac{p_3}{2}, \frac{p_3}{2}, \frac{p_1}{2} \right). \]
In order to obtain the same Dolgachev numbers for $F$, $F(1,0,z)$ must be non-zero if $z \neq 0$. Therefore the polynomial $F$ must contain 3 monomials involving $y$. This leaves us with the only additional possibility
\[ F(x,y,z)= a_1x^{\frac{p_1}{2}}z^{\frac{p_3}{2}}+a_2y^2 + a_3yz^{\frac{p_3}{2}} +a_4x^{\frac{p_1}{2}}y. \]
\end{proof}

We associate to these matrices a pair of polynomials as follows. We observe that the kernel of the matrix $E^T$ is either generated by the vector $(1,1,0,-2)^T$ or by the vector $(1,1,-1,-1)^T$. The second case occurs precisely for the matrices of type ${\rm IIB}^\sharp$ and ${\rm IV}^\sharp$.
Let $R:=\CC[x,y,z,w]$. In the first case, there exists a $\ZZ$-graded structure on $R$ given by the $\CC^\ast$-action
\[ \lambda \ast (x,y,z,w) = (\lambda x, \lambda y, z , \lambda^{-2} w) \quad \mbox{for } \lambda \in \CC^\ast.
\]
In the second case, there exists a $\ZZ$-graded structure on $R$ given by the $\CC^\ast$-action
\[ \lambda \ast (x,y,z,w) = (\lambda x, \lambda y, \lambda^{-1}z , \lambda^{-1} w) \quad \mbox{for } \lambda \in \CC^\ast.
\]
Let $R= \bigoplus_{i \in \ZZ} R_i$ be the decomposition of $R$ according to one of these $\ZZ$-gradings. Let $E^T$ be the transposed matrix.  We associate to this the polynomial 
\[\widetilde{f}(x,y,z,w) := x^{E_{11}}y^{E_{21}}z^{E_{31}}w^{E_{41}}+ x^{E_{12}}y^{E_{22}}z^{E_{32}}w^{E_{42}} +  x^{E_{13}}y^{E_{23}}z^{E_{33}}w^{E_{43}}.
\]
In the first case, we have $\widetilde{f} \in R_0=\CC[x^2w,y^2w,z,xyw]$.
Let
\[ X:=x^2w, \quad Y:=y^2w, \quad Z:=z, \quad W:=xyw. 
\]
In these new coordinates, we obtain a pair of polynomials
\[
\widetilde{\ff}_1(X,Y,Z,W) = XY-W^2, \quad \widetilde{\ff}_2(X,Y,Z,W) = \widetilde{f}(X,Y,Z,W).
\]

In the second case, we have  $\widetilde{f} \in R_0=\CC[xw,yz,xz,yw]$. Let
\[ X:=xw, \quad Y:=yz, \quad Z:=xz \quad W:=yw. 
\]
In these new coordinates, we obtain a pair of polynomials
\[
\widetilde{\ff}_1(X,Y,Z,W) = XY-ZW, \quad \widetilde{\ff}_2(X,Y,Z,W) = \widetilde{f}(X,Y,Z,W).
\]

Now we choose for each of the matrices $E$ special values $a_1, a_2, a_3, a_4$ such that the corresponding polynomial $F$ has a non-isolated singularity. We denote this polynomial by $\ff$.
We summarize the results in Table~\ref{TabHSCIS}.
\begin{table}[h]
\begin{center}
\begin{tabular}{|c|c|c|}
\hline
Type   & $\ff$ & $(\widetilde{\ff}_1, \widetilde{\ff}_2)$\\
\hline
I &     $x^{p_1}+y^{p_2}+z^{p_3} -2 x^{\frac{p_1}{2}}y^{\frac{p_2}{2}}$  & $\left\{ \begin{array}{c} XY-W^2\\ X^{\frac{p_1}{2}}+Y^{\frac{p_2}{2}}+ Z^{p_3} \end{array}\right\}$ \\
IIA &   $x^{p_2}+xy^{\frac{p_3}{p_2}}+z^{p_1}  -2 x^{\frac{p_2+1}{2}}y^{\frac{p_3}{2p_2}}$ & 
$\left\{ \begin{array}{c} XY-W^2\\ XW+ Y^{\frac{p_3}{2p_2}}+Z^{p_1}\end{array}\right\}$\\
IIB &   $x^{p_1} + y^{p_2} + yz^{\frac{p_3}{p_2}} -2x^{\frac{p_1}{2}}y^{\frac{p_2}{2}}$ & $\left\{ \begin{array}{c} XY-W^2\\ X^{\frac{p_1}{2}}+Y^{\frac{p_2}{2}}Z+Z^{\frac{p_3}{p_2}} \end{array}\right\}$\\
${\rm IIB}^\sharp$ &  $-x^{\frac{p_1}{2}}z^{\frac{p_3}{2}}+y^2 + yz^{\frac{p_3}{2}} -x^{\frac{p_1}{2}}y$ & $\left\{ \begin{array}{c} XY-ZW\\ X^{\frac{p_1}{2}}+YW+Z^{\frac{p_3}{2}} \end{array}\right\}$\\
III &   $x^{q_2+1}y + xy^{q_3+1}+z^{p_1} - 2x^{\frac{q_2}{2}+1}y^{\frac{q_3}{2}+1}$ & $\left\{ \begin{array}{c} XY-W^2\\ (X^{\frac{q_2}{2}}+Y^{\frac{q_3}{2}})W+Z^{p_1} \end{array}\right\}$\\
IV &   $x^{p_1} + xy^{\frac{p_2}{p_1}} +yz^{\frac{p_3}{p_2}} -2x^{\frac{p_1+1}{2}}y^{\frac{p_2}{2p_1}}$ & $\left\{ \begin{array}{c} XY-W^2\\ X^{\frac{p_1-1}{2}}W+Y^{\frac{p_2}{2p_1}}Z+ Z^{\frac{p_3}{p_2}} \end{array}\right\}$\\
${\rm IV}^\sharp$ &  $-x^{\frac{p_1-1}{2}}z^{\frac{p_3}{p_2}}+xy^2+yz^{\frac{p_3}{p_2}}-x^{\frac{p_1+1}{2}}y$ & $\left\{ \begin{array}{c} XY-ZW\\ X^{\frac{p_1-1}{2}}W+YW+Z^{\frac{p_3}{p_2}}\end{array}\right\}$\\
\hline
\end{tabular}
\end{center}
\caption{Correspondence between polynomials $\ff$ and pairs of polynomials $(\widetilde{\ff}_1, \widetilde{\ff}_2)$} \label{TabHSCIS}
\end{table}

\section{Virtual singularities}
We now associate other equations to the polynomials from above. For each type, consider the polynomial  $\ff$ from Table~\ref{TabHSCIS} and assume that the conditions indicated in Table~\ref{TabCond} are satisfied. In all cases except ${\rm IIB}^\sharp$ and ${\rm IV}^\sharp$, the polynomial $\ff(x,y,z)$ is of the form
\[ 
\ff(x,y,z)= u(x,y,z)+v(x,y,z)(x-y^e)^2
\]
or
\[ 
\ff(x,y,z)= u(x,y,z)+v(x,y,z)(y-x^e)^2
\]
for some monomials $u(x,y,z)$ and $v(x,y,z)$ and some integer $e \geq 2$. We
consider the cusp singularity $\ff(x,y,z)-xyz$ and perform the coordinate change $x \mapsto x+y^e$ or $y \mapsto y+x^e$ respectively. The corresponding coordinate change is indicated in Table~\ref{TabCond}. Then $\ff(x,y,z)-xyz$ is transformed to $\hh(x,y,z)-xyz$ where $\hh(x,y,z)$ is indicated in the last column of Table~\ref{TabCond}. 
\begin{table}[h]
\begin{center}
\begin{tabular}{|c|c|c|c|}
\hline
Type  & Conditions & Coord.change & $\hh(x,y,z)$ \\
\hline
I &  $p_2=2$ & $y \mapsto y+x^{\frac{p_1}{2}}$ &  $-x^{\frac{p_1}{2}+1}z +y^2+z^{p_3}$\\
IIA & $p_2=3$ &  $x \mapsto x+y^{\frac{p_3}{6}}$ &  $-y^{\frac{p_3}{6}+1}z  +z^{p_1} +x^3+x^2y^{\frac{p_3}{6}}$\\
IIA & $\frac{p_3}{p_2}=2$ & $y \mapsto y+x^{\frac{p_2-1}{2}}$ & $-x^{\frac{p_2+1}{2}}z+z^{p_1}+xy^2$\\
IIB & $p_1=2$ & $x \mapsto x+y^{\frac{p_2}{2}}$ & $-y^{\frac{p_2}{2}+1}z + x^2+yz^{\frac{p_3}{p2}}$ \\
IIB &  $p_2=2$ & $y \mapsto y+x^{\frac{p_1}{2}}$ & $-x^{\frac{p_1}{2}+1}z+ y^2 + yz^{\frac{p_3}{2}} +x^{\frac{p_1}{2}}z^{\frac{p_3}{2}}$ \\
${\rm IIB}^\sharp$   & $p_2=2$ & $y \mapsto y+x^{\frac{p_1}{2}}$ & $-x^{\frac{p_1}{2}+1}z+y^2 + yz^{\frac{p_3}{2}} +x^{\frac{p_1}{2}}y$ \\
III  & $q_2=2$ & $x \mapsto x+y^{\frac{q_3}{2}}$ & $-y^{\frac{q_3}{2}+1}z +z^{p_1}  + x^3y +x^2y^{\frac{q_3}{2}+1}$ \\
${\rm IV}_1$   & $p_1=3$ & $x \mapsto x+y^{\frac{p_2}{6}}$ & $-y^{\frac{p_2}{6}+1}z+x^3+yz^{\frac{p_3}{p_2}}+x^2y^{\frac{p_2}{6}}$\\
${\rm IV}_2$ & $\frac{p_2}{p_1}=2$ & $y \mapsto y+x^{\frac{p_1-1}{2}}$ & $-x^{\frac{p_1+1}{2}}z+ xy^2 +yz^{\frac{p_3}{p_2}}+x^{\frac{p_1-1}{2}}z^{\frac{p_3}{p_2}}$ \\
${\rm IV}_2^\sharp$ & $\frac{p_2}{p_1}=2$ & $y \mapsto y+x^{\frac{p_1-1}{2}}$ & $-x^{\frac{p_1+1}{2}}z+ xy^2 +yz^{\frac{p_3}{p_2}}+x^{\frac{p_1+1}{2}}y$ \\
\hline
\end{tabular}
\end{center}
\caption{Conditions and transformations} \label{TabCond}
\end{table}

By inspection of Table~\ref{TabCond}, we see that some of the polynomials $\hh$ have 4 monomials and others only 3. We restrict our consideration to the cases where the polynomial $\hh$ has 4 monomials. These cases are listed in Table~\ref{TabVirt}.  The singularities defined by the polynomials $\hh(x,y,z)$ will be called {\em virtual singularities}. We consider the duality between the virtual singularities on one side and the complete intersection singularities on the other side. 

Let 
\[ \hh(x,y,z) = \sum_{i=1}^4 a_i x^{A_{i1}}y^{A_{i2}}z^{A_{i3}}
\]
be the polynomial defining a virtual singularity and let ${\rm Supp}(\hh)=\{ (A_{i1}, A_{i2}, A_{i3}) \in \ZZ^3 \, | \, i=1, \ldots ,4\}$.
Let $\Delta_\infty(\hh)$ be the Newton polygon of $\hh$ at infinity \cite{Ko}, i.e.\ $\Delta_\infty(\hh)$ is the convex closure in $\RR^n$ of ${\rm Supp}(\hh) \cup \{ 0 \}$. The Newton polygon $\Delta_\infty(\hh)$ has two faces which do not contain the origin.  
Call these faces $\Sigma_1$ and $\Sigma_2$. Let $I_k:= \{i \in \{1, \ldots, 4\} \, | \, (A_{i1}, A_{i2}, A_{i3}) \in \Sigma_k \}$, $k=1,2$, and let
\[
\hh_k= \sum_{i \in I_k} a_i x^{A_{i1}}y^{A_{i2}}z^{A_{i3}}.
\]
Then $\hh_k$ is an invertible polynomial with a non-isolated singularity at the origin. The polynomials $\hh_1$ and $\hh_2$ are listed in Table~\ref{TabVirt}. Their canonical systems of weights are reduced. One of the systems of weights of $\hh_1$ and $\hh_2$ coincides with the reduced system of weights of the non-degenerate invertible polynomial $f$ we started with. Let the numbering be chosen that this is the system of weights of $\hh_2$. The systems of weights are listed in Table~\ref{TabVirtWeight}.
\begin{table}[h]
\begin{center}
\begin{tabular}{|c|c|c|c|}
\hline
Type  & & $\hh_1(x,y,z)$  & $\hh_2(x,y,z)$\\
\hline
IIA & $p_2=3$ & $-y^{\frac{p_3}{6}+1}z+z^{p_1} +x^2y^{\frac{p_3}{6}}$ & $z^{p_1} +x^3 +x^2y^{\frac{p_3}{6}}$ \\
IIB & $p_2=2$ &  $-x^{\frac{p_1}{2}+1}z + y^2 +x^{\frac{p_1}{2}}z^{\frac{p_3}{2}}$  & $y^2 + yz^{\frac{p_3}{2}} +x^{\frac{p_1}{2}}z^{\frac{p_3}{2}}$ \\
${\rm IIB}^\sharp$  & $p_2=2$ & $-x^{\frac{p_1}{2}+1}z+ yz^{\frac{p_3}{2}} +x^{\frac{p_1}{2}}y$  & $y^2 + yz^{\frac{p_3}{2}} +x^{\frac{p_1}{2}}y$ \\
III  & $q_2=2$ &  $-y^{\frac{q_3}{2}+1}z + z^{p_1} +x^2y^{\frac{q_3}{2}+1}$  & $z^{p_1} + x^3y +x^2y^{\frac{q_3}{2}+1}$ \\
${\rm IV}_1$  & $p_1=3$ & $-y^{\frac{p_2}{6}+1}z+yz^{\frac{p_3}{p_2}}+x^2y^{\frac{p_2}{6}}$  & $x^3+yz^{\frac{p_3}{p_2}}+x^2y^{\frac{p_2}{6}}$  \\
${\rm IV}_2$  & $\frac{p_2}{p_1}=2$ & $-x^{\frac{p_1+1}{2}}z +xy^2+x^{\frac{p_1-1}{2}}z^{\frac{p_3}{p_2}}$   & $xy^2 +yz^{\frac{p_3}{p_2}}+x^{\frac{p_1-1}{2}}z^{\frac{p_3}{p_2}}$ \\
${\rm IV}_2^\sharp$ & $\frac{p_2}{p_1}=2$& $-x^{\frac{p_1+1}{2}}z +yz^{\frac{p_3}{p_2}}+x^{\frac{p_1+1}{2}}y$  & $xy^2 +yz^{\frac{p_3}{p_2}}+x^{\frac{p_1+1}{2}}y$ \\
\hline
\end{tabular}
\end{center}
\caption{Virtual singularities} \label{TabVirt}
\end{table}

\begin{table}[h]
\begin{center}
\begin{tabular}{|c|c|c|}
\hline
Type   & System of weights of $\hh_1$ & System of weights of $\hh_2$ \\
\hline
IIA  &  $\left(p_1+\frac{p_3}{6},2p_1-2,\frac{p_3}{3}+2;p_1 \left(\frac{p_3}{3}+2 \right)\right)$ &  $\left(  \frac{p_1p_3}{6}, p_1,\frac{p_3}{2}; \frac{p_1p_3}{2} \right)$ \\
IIB  & $(p_3-2, \frac{p_1p_3}{4}+\frac{p_3}{2}-\frac{p_1}{2},2; \frac{p_1p_3}{2}+p_3-p_1)$ &  $(\frac{p_3}{2}, \frac{p_1p_3}{4}, \frac{p_1}{2}; \frac{p_1p_3}{2})$\\
${\rm IIB}^\sharp$   & $(\frac{p_3}{2}, \frac{p_1}{2}+\frac{p_3}{2}, \frac{p_1}{2}; \frac{p_1p_3}{4}+\frac{p_3}{2}+\frac{p_1}{2})$ & $(\frac{p_3}{2}, \frac{p_1p_3}{4}, \frac{p_1}{2}; \frac{p_1p_3}{2})$ \\
III   & $(\frac{q_3}{2}+1,2p_1-2,q_3+2;p_1(q_3+2))$ & $(\frac{q_3}{2}p_1,p_1,3\frac{q_3}{2}+1;p_1(3\frac{q_3}{2}+1))$ \\
${\rm IV}_1$    & $(\frac{p_2}{6} +\frac{p_3}{p_2}-1,2\frac{p_3}{p_2}-2,\frac{p_2}{3};\frac{p_3}{3}+2\frac{p_3}{p_2}-2)$ & $(\frac{p_3}{6},\frac{p_3}{p_2}, \frac{p_2}{2}-1;\frac{p_3}{2})$ \\
${\rm IV}_2$  &  $(2\frac{p_3}{p_2}-2, \frac{p_3}{4}- \frac{p_3}{2p_2} - \frac{p_1}{2}+\frac{3}{2},2;\frac{p_3}{2}+ \frac{p_3}{p_2}-p_1+1)$ & $(\frac{p_3}{p_2}, (p_1-1)\frac{p_3}{2p_2}, \frac{p_1+1}{2}; \frac{p_3}{2})$ \\
${\rm IV}_2^\sharp$ & $(\frac{p_3}{p_2}, \frac{p_1+1}{2},  \frac{p_1+1}{2}; \frac{p_3}{4}+\frac{p_3}{2p_2}+ \frac{p_1}{2} + \frac{1}{2})$ &   $(\frac{p_3}{p_2}, (p_1-1)\frac{p_3}{2p_2}, \frac{p_1+1}{2}; \frac{p_3}{2})$ \\
\hline
\end{tabular}
\end{center}
\caption{Systems of weights corresponding to the virtual singularities} \label{TabVirtWeight}
\end{table}

The dual complete intersection singularity defined by $\widetilde{\ff}_1=\widetilde{\ff}_2=0$ is weighted homogeneous.
We list the systems of weights of these complete intersection singularities in Table~\ref{TabCISWeight}. \begin{table}[h]
\begin{center}
\begin{tabular}{|c|c|}
\hline
Type   & System of weights \\
\hline
IIA  & $(p_1(\frac{p_3}{3}-1), 3p_1, \frac{p_3}{2}, \frac{p_1}{2}(\frac{p_3}{3}+2);p_1(\frac{p_3}{3}+2) ,\frac{p_1p_3}{2})$ \\
IIB  & $(p_3,\frac{p_1p_3}{2}-p_1,p_1,\frac{p_1p_3}{4}+\frac{p_3}{2}-\frac{p_1}{2}; \frac{p_1p_3}{2}+p_3-p_1, \frac{p_1p_3}{2})$\\
${\rm IIB}^\sharp$   & $(p_3, \frac{p_1p_3}{4}+\frac{p_1}{2}-\frac{p_3}{2},p_1,\frac{p_1p_3}{4}+\frac{p_3}{2}-\frac{p_1}{2}; \frac{p_1p_3}{4}+\frac{p_3}{2}+\frac{p_1}{2}, \frac{p_1p_3}{2})$ \\
III   & $( p_1q_3,2p_1, 3\frac{q_3}{2}+1, (\frac{q_3}{2}+1)p_1;p_1(q_3+2), p_1(3\frac{q_3}{2}+1))$ \\
${\rm IV}_1$    & $(\frac{p_3}{3}-\frac{p_3}{p_2}+1,3(\frac{p_3}{p_2}-1), \frac{p_2}{2}, \frac{p_3}{6}+\frac{p_3}{p_2}-1 ;\frac{p_3}{3}+2\frac{p_3}{p_2}-2,\frac{p_3}{2})$ \\
${\rm IV}_2$  &  $(\frac{p_3}{p_2}+1,\frac{p_3}{2}-p_1,p_1,\frac{p_3}{4}+ \frac{p_3}{2p_2}-\frac{p_1}{2}+\frac{1}{2};\frac{p_3}{2}+ \frac{p_3}{p_2}-p_1+1, \frac{p_3}{2})$ \\
${\rm IV}_2^\sharp$ & $(\frac{p_3}{p_2}+1,\frac{p_3}{4}- \frac{p_3}{2p_2}+\frac{p_1}{2}-\frac{1}{2},p_1,\frac{p_3}{4}+ \frac{p_3}{2p_2}-\frac{p_1}{2}+\frac{1}{2}; \frac{p_3}{4}+\frac{p_3}{2p_2}+ \frac{p_1}{2} + \frac{1}{2}, \frac{p_3}{2})$ \\
\hline
\end{tabular}
\end{center}
\caption{Systems of weights of the pairs $(\widetilde{\ff}_1,\widetilde{\ff}_2)$} \label{TabCISWeight}
\end{table}
It turns out that the degrees of the systems of weights of $\hh_1$ and $\hh_2$ coincide with the degrees of the two polynomials $\widetilde{\ff}_1$ and $\widetilde{\ff}_2$ respectively.

\section{Dolgachev and Gabrielov numbers} \label{SectDolGab}
We shall now define Dolgachev and Gabrielov numbers for the polynomials $\hh$ and and the pairs of polynomials $(\widetilde{\ff}_1,\widetilde{\ff}_2)$ occurring in our duality.

We first define these numbers for the pairs $(\widetilde{\ff}_1,\widetilde{\ff}_2)$.
Let $X_{\widetilde{\ff}_1,\widetilde{\ff}_2} \subset \CC^4$ be the weighted homogeneous complete intersection in $\CC^4$ defined by the  two equations $\widetilde{\ff}_1(W,X,Y,Z)=\widetilde{\ff}_2(W,X,Y,Z)=0$, where $\widetilde{\ff}_1(W,X,Y,Z)=XY-W^2$ or $\widetilde{\ff}_1(W,X,Y,Z)=XY-ZW$. 

\begin{definition} Let $\calC_{\widetilde{\ff}_1,\widetilde{\ff}_2} := [(X_{\widetilde{\ff}_1,\widetilde{\ff}_2} \setminus \{0\})/\CC^\ast]$. Then $\calC_{\widetilde{\ff}_1,\widetilde{\ff}_2}$ is a smooth projective curve with three isotropic points of orders $\alpha_1, \alpha_2, \alpha_3$. We call these numbers the {\em Dolgachev numbers} of the pair $(\widetilde{\ff}_1,\widetilde{\ff}_2)$.
\end{definition}

The Gabrielov numbers are defined similarly as in the hypersurface case. We consider the complete intersection singularity $(X',0)$ defined by
\[  \left\{ \begin{array}{c} \widetilde{\ff}_1(W,X,Y,Z), \\
\widetilde{\ff}_2(W,X,Y,Z) -ZW. \end{array} \right. \]
As in \cite{ET1} one can show that one can find a holomorphic change of coordinates such that the singularity $(X',0)$ is also given by equations of the form
\[  \left\{ \begin{array}{c} XY-Z^{\gamma_1} - W^{\gamma_2},\\
 X^{\gamma_3} + Y^{\gamma_4} -ZW.
 \end{array} \right. \]
This means that $(X',0)$ is a cusp singularity of type $T^2_{\gamma_1,\gamma_3, \gamma_2, \gamma_4}$ in the notation of \cite[3.1]{E1}.

\begin{definition} The {\em Gabrielov numbers} of the pair $(\widetilde{\ff}_1,\widetilde{\ff}_2)$ are the numbers $(\gamma_1,\gamma_2; \gamma_3, \gamma_4)$.
\end{definition}

The Dolgachev and Gabrielov numbers for the pairs $(\widetilde{\ff}_1,\widetilde{\ff}_2)$ of Table~\ref{TabHSCIS} are indicated in Table~\ref{TabCISDolGab}.
\begin{table}[h]
\begin{center}
\begin{tabular}{|c|c|c|}
\hline
Type  & Dolgachev & Gabrielov\\
\hline
I &  $p_1,p_2,p_3$ & $2, 2p_3-2; \frac{p_1}{2}, \frac{p_2}{2}$  \\
IIA  & $p_1,p_2,\left( \frac{p_3}{p_2}-1 \right)p_1$ & $2, 2p_1-2; \frac{p_1(p_2-1)}{2}, \frac{p_3}{2p_2}$ \\
IIB  & $p_1,p_2,\left( \frac{p_3}{p_2}-1 \right)p_1$ & $2, 2\frac{p_3}{p_2}-2; \frac{p_1}{2}, \frac{p_1(p_2-1)}{2}$ \\
${\rm IIB}^\sharp$  & $\frac{p_3}{2} \left( \frac{p_1}{2}-1 \right)+ \frac{p_1}{2}, 2, \frac{p_1}{2} \left( \frac{p_3}{2}-1 \right) + \frac{p_3}{2}$  & $\frac{p_3}{2}, \frac{p_1}{2}; \frac{p_3}{2},\frac{p_1}{2}$ \\
III  & $p_1,p_1q_2, p_1q_3$ & $2, 2p_1-2; p_1, \frac{q_3}{2}p_1$  \\
IV  & $p_1,\left( \frac{p_3}{p_2}-1 \right)p_1, \frac{p_3}{p_1}-\frac{p_3}{p_2}+1$ & $2, 2\frac{p_3}{p_2}-2; \frac{1}{2} \frac{p_3}{p_2}(p_1-1), \frac{1}{2}(p_2-p_1+1)$ \\
${\rm IV}^\sharp$   & $\frac{p_1-1}{2} \left( \frac{p_3}{p_2}+1 \right), \frac{p_3}{p_2}+1, \frac{p_1+1}{2} \left( \frac{p_3}{p_2}-1 \right)+1$ & $\frac{p_3}{p_2},\frac{1}{2}(p_1+1); \frac{p_3}{p_2}, \frac{1}{2} \frac{p_3}{p_2}(p_1-1)$ \\
\hline
\end{tabular}
\end{center}
\caption{Pairs $(\widetilde{\ff}_1,\widetilde{\ff}_2)$: Dolgachev and Gabrielov numbers} \label{TabCISDolGab}
\end{table}

Now let $\hh$ be the polynomial of Table~\ref{TabVirt} defining a virtual singularity.
The Gabrielov numbers of $\hh$ are defined as in \cite{ET1}. Namely, we consider the polynomial
$\hh(x,y,z) -xyz$.
As in \cite{ET1}, one can show that one can find a holomorphic change of coordinates at the origin such that this polynomial becomes
\[ x^{\gamma_1} + y^{\gamma_2} + z^{\gamma_3} - xyz .\]
We define the {\em Gabrielov numbers} of $\hh$ to be the triple $(\gamma_1, \gamma_2, \gamma_3)$. 

The Dolgachev numbers of the polynomial $\hh$ are defined as follows. We associated to $\hh$ two weighted homogeneous polynomials $\hh_1$ and $\hh_2$. Let $i=1,2$ and let
$V_i:= \{ (x,y,z) \in \CC^3 \, | \, \hh_i(x,y,z)=0 \}$.
We consider the $\CC^\ast$-action on $V_i$ given by the system of weights of $\hh_i$ (see Table~\ref{TabVirtWeight}). We consider the exceptional orbits (i.e.\ orbits with a non-trivial isotropy group) of this action. We distinguish between two cases:
\begin{itemize}
\item[(A)] $V_i$ contains a coordinate hyperplane.
\item[(B)] $V_i$ does not contain a coordinate hyperplane.
\end{itemize}
In case (A) we consider those exceptional orbits which are not contained in the coordinate hyperplane which is contained in $V_i$. In case (B) we consider those exceptional orbits which do not coincide with the singular locus of $V_i$. We call these the {\em principal} orbits. It turns out that in all cases we have exactly two principal orbits. 

\begin{example} (a) IIA ($p_2=3$): $\hh_1(x,y,z)=-y^{\frac{p_3}{6}+1}z+z^{p_1} +x^2y^{\frac{p_3}{6}}$ with the system of weights $\left(p_1+\frac{p_3}{6},2p_1-2,\frac{p_3}{3}+2;p_1 \left(\frac{p_3}{3}+2 \right)\right)$. The exceptional orbits are:\\
\begin{tabular}{cl}
$y=z=0$ & singular line\\
$x=-y^{\frac{p_3}{6}+1}z+z^{p_1}=0$ & order of isotropy group: 2 \\
$x=z=0$ & order of isotropy group: $2p_1-2$
\end{tabular}

(b) ${\rm IV}_2$ ($\frac{p_2}{p_1}=2$): $\hh_1(x,y,z)=-x^{\frac{p_1+1}{2}}z +xy^2+x^{\frac{p_1-1}{2}}z^{\frac{p_3}{p_2}}=x(-x^{\frac{p_1-1}{2}}z +y^2+x^{\frac{p_1-3}{2}}z^{\frac{p_3}{p_2}})$ with the system of weights $(2\frac{p_3}{p_2}-2, \frac{p_3}{4}- \frac{p_3}{2p_2} - \frac{p_1}{2}+\frac{3}{2},2;\frac{p_3}{2}+ \frac{p_3}{p_2}-p_1+1)$. The exceptional orbits not contained in the hyperplane $x=0$ are:\\
\begin{tabular}{cl}
$y=z=0$ & order of isotropy group: $2\frac{p_3}{p_2}-2$\\
$y=-x^{\frac{p_1+1}{2}}z+x^{\frac{p_1-1}{2}}z^{\frac{p_3}{p_2}}=0$ & order of isotropy group: 2
\end{tabular}
\end{example}

\begin{definition} The {\em Dolgachev numbers} of $\hh$ are the numbers $\alpha_1, \alpha_2; \alpha_3, \alpha_4$ where $\alpha_1, \alpha_2$ and $\alpha_3, \alpha_4$ are the orders of the isotropy groups of the principal exceptional orbits of $\hh_1$ and $\hh_2$ respectively.
\end{definition}

We list the Dolgachev and Gabrielov numbers of the polynomials $\hh$ corresponding to the virtual singularities in Table~\ref{TabHSDolGab}.
\begin{table}[h]
\begin{center}
\begin{tabular}{|c|c|c|}
\hline
Type  & Dolgachev & Gabrielov\\
\hline
IIA &  $2, 2p_1-2; p_1, \frac{p_3}{6}$ & $p_1,3,\left( \frac{p_3}{3}-1 \right)p_1$\\
IIB &  $2, 2\frac{p_3}{p_2}-2; \frac{p_1}{2}, \frac{p_1}{2}$ & $p_1,2,\left( \frac{p_3}{2}-1 \right)p_1$ \\
${\rm IIB}^\sharp$   & $\frac{p_3}{2}, \frac{p_1}{2}; \frac{p_3}{2},\frac{p_1}{2}$ & $\frac{p_3}{2} \left( \frac{p_1}{2}-1 \right)+ \frac{p_1}{2}, 2, \frac{p_1}{2} \left( \frac{p_3}{2}-1 \right) + \frac{p_3}{2}$ \\
III  & $2, 2p_1-2; p_1, \frac{q_3}{2}p_1$ & $p_1,2p_1, q_3p_1$ \\
${\rm IV}_1$ & $2, 2\frac{p_3}{p_2}-2; \frac{p_3}{p_2}, \frac{1}{2}(p_2-2)$ & $3,3(\frac{p_3}{p_2}-1),\frac{p_3}{3}-\frac{p_3}{p_2}+1$ \\
${\rm IV}_2$   & $2, 2\frac{p_3}{p_2}-2; \frac{1}{2} \frac{p_3}{p_2}(p_1-1), \frac{1}{2}(p_1+1)$ & $p_1,\left( \frac{p_3}{p_2}-1 \right)p_1, \frac{p_3}{p_1}-\frac{p_3}{p_2}+1$\\
${\rm IV}_2^\sharp$ &  $\frac{p_3}{p_2},\frac{1}{2}(p_1+1); \frac{p_3}{p_2}, \frac{1}{2} \frac{p_3}{p_2}(p_1-1)$  & $\frac{p_1-1}{2} \left( \frac{p_3}{p_2}+1 \right), \frac{p_3}{p_2}+1, \frac{p_1+1}{2} \left( \frac{p_3}{p_2}-1 \right)+1$\\
\hline
\end{tabular}
\end{center}
\caption{Virtual singularities: Dolgachev and Gabrielov numbers} \label{TabHSDolGab}
\end{table}

\section{Strange duality}

Comparing Table~\ref{TabHSDolGab} with Table~\ref{TabCISDolGab}, we obtain the following result.

\begin{theorem} \label{thm:duality}
The Gabrielov numbers of the polynomial $\hh$ corresponding to a virtual singularity  coincide with the Dolgachev numbers of the dual pair $(\widetilde{\ff}_1,\widetilde{\ff}_2)$ and, vice versa, the Gabrielov numbers of a pair $(\widetilde{\ff}_1,\widetilde{\ff}_2)$ coincide with the Dolgachev numbers of the dual polynomial $\hh$.
\end{theorem}

Let $f_1$, \dots , $f_k$ be quasihomogeneous functions on $\CC^n$  of
degrees $d_1$, \dots , $d_k$ with respect to weights $w_1$, \dots, $w_n$. Here $w_1$, \dots,
$w_n$ are positive integers with $\gcd(w_1, \ldots, w_n)=1$, 
$f_j(\lambda^{w_1}x_1, \ldots, \lambda^{w_n}x_n)=
\lambda^{d_j}f_j(x_1, \ldots, x_n)$, $\lambda\in\CC$.
We suppose that the equations $f_1=f_2= \ldots = f_k=0$ define a complete intersection $X$ in
$\CC^n$. There is a natural $\CC^\ast$-action on the space $\CC^n$ defined by 
$\lambda\ast(x_1, \ldots, x_n)=
(\lambda^{w_1}x_1, \ldots, \lambda^{w_n}x_n)$, $\lambda \in \CC^\ast$.

Let
$A=\CC[x]/(f_1, \ldots , f_k)$ be the coordinate ring of $X$. There is a 
natural grading on the ring $A$: $A_s$ is the set of functions
$g\in A$ such that $g(\lambda\ast x)=\lambda^s g(x)$. Let
$P_X(t)=\sum_{s=0}^\infty \dim A_s \cdot t^s$ be the Poincar\'e series
of the graded algebra $A=\oplus_{s=0}^\infty A_s$. One has
\begin{equation}
P_X(t) = \frac{ \prod_{j=1}^k (1-t^{d_j}) }{\prod_{i=1}^n (1-t^{w_i})}.
\label{eq:poinc}
\end{equation}

For $0 \leq j \leq k$, let $X^{(j)}$ be the complete intersection given by the equations $f_1=
\ldots = f_j=0$ ($X^{(0)}=\CC^n$, $X^{(k)}=X$). The restriction of the function $f_j$ ($j=1,
\ldots , k$) to the variety $X^{(j-1)}$ defines a locally trivial fibration $X^{(j-1)} \setminus
X^{(j)} \to \CC^\ast$. Let $V^{(j)} = f_j^{-1}(1) \cap X^{(j-1)}$ be the (Milnor) fibre of this
fibration (the fibre $V^{(j)}$ is not necessarily smooth) and $\varphi^{(j)} : V^{(j)} \to V^{(j)}$  be
the classical monodromy transformation of it. For a map $\varphi: Z \to Z$ of a topological space $Z$,
let
$\zeta_{\varphi}(t)$ be its {\em zeta function}
$$\zeta_{\varphi}(t)=\prod_{p\ge0}
\left\{\det \left( \mbox{id} -t\cdot {\varphi}_\ast\vert_{H_p(Z;\CC)}\right)\right\}^{(-1)^p}.$$
If, in the definition, we use the actions of the operators ${\varphi}_\ast$ on the homology groups
$\overline{H}_p(Z;\ZZ)$ reduced modulo a point, we get the {\em reduced} zeta function
$$
\overline\zeta_{\varphi}(t)  =  \frac{\zeta_{\varphi}(t)}{(1-t)}. 
$$
Let
$$\overline\zeta_{X,j}(t) := \overline\zeta_{{\varphi}^{(j)}}(t).$$
If both  $X^{(j)}$ and $X^{(j-1)}$ have isolated singularities at the origin
then $\overline{H}_p(V^{(j)};\ZZ)$ is non-trivial only for $p=n-j$ and therefore, if $n-j \ge 1$,
$$\left(\overline\zeta_{X,j}(t)\right)^{(-1)^{n-j}}=
\det \left(\mbox{id}-t\cdot {\varphi}^{(j)}_\ast\vert_{H_{n-j}(V^{(j)};\CC)}\right)$$
is the characteristic polynomial of the classical monodromy
operator ${\varphi}_\ast^{(j)}$.

One can show that $({\varphi}^{(j)}_\ast)^{d_j} = \mbox{id}$ and therefore 
$\overline\zeta_{X,j}(t)$ can be written in the form
$$
\prod_{\ell\vert d_j}(1-t^\ell)^{\alpha_\ell}, \ \alpha_\ell\in\ZZ.
$$
Following K.~Saito \cite{S1, S2}, we define the Saito dual to $\overline\zeta_{X,j}(t)$ to be the
rational function 
$$
\overline\zeta_{X,j}^\ast(t)=\prod_{m\vert d_j}(1-t^m)^{-\alpha_{(d_j/m)}}
$$
(note that different degrees $d_j$ are used for different $j$).

Let $Y^{(k)}=(X^{(k)} \setminus\{0\})/\CC^\ast$ be the space of orbits of the
$\CC^\ast$-action on $X^{(k)}\setminus\{0\}$ and
$Y^{(k)}_m$ be the set of orbits for which the isotropy group is the
cyclic group of order $m$. Let
$$\mbox{Or}_X(t) := \prod_{m \geq 1} (1-t^m)^{\chi(Y_m^{(k)})}$$
be the product of cyclotomic polynomials with exponents corresponding to the partition of the
complete intersection $X=X^{(k)}$ into parts of different orbit types; here $\chi(Z)$ denotes the
Euler characteristic of a topological space $Z$.

Let $(X,0)$ be the virtual hypersurface singularity defined by $\hh=0$ and $(\widetilde{X},0)$ be the dual complete intersection singularity given by the equations $\widetilde{\ff}_1=\widetilde{\ff}_2=0$ according to Theorem~\ref{thm:duality}. The function $\mbox{Or}_{\widetilde{X}}(t)$ is equal to  the polynomial
$$\mbox{Or}_{\widetilde{X}}(t) := \prod_{k=1}^3 (1-t^{\alpha_k}) \cdot (1-t)^{-1},$$
where $\alpha_1, \alpha_2, \alpha_3$ are the Dolgachev numbers of the pair  $(\widetilde{\ff}_1,\widetilde{\ff}_2)$, see Sect.~\ref{SectDolGab}.

Finally, let $\zeta_X(t)$ be the zeta function of the monodromy at infinity of $\hh$ and
$$
\overline\zeta_X(t)  =  \frac{\zeta_X(t)}{(1-t)}. 
$$
be the reduced zeta function of $\hh$. 

\begin{theorem} \label{thm:ZetaHS}
Under the conditions of Table~\ref{TabVirt}  we have
\[
\overline{\zeta}_X(t)= P_{\widetilde{X}}(t) \cdot {\rm Or}_{\widetilde{X}}(t).
\]
\end{theorem}

\begin{proof}
The zeta function $\zeta_X(t)$ can be computed from the Newton polygon of $\hh$ at infinity by \cite{LS}. From Tables~\ref{TabVirtWeight}, \ref{TabCISWeight}, and \ref{TabCISDolGab} we can derive the formula.
\end{proof}

From \cite{EG} we get the following corollary:

\begin{corollary} 
Under the conditions of Table~\ref{TabVirt}  we have 
\[
\overline{\zeta}_X(t) = \overline{\zeta}^\ast_{\widetilde{X},1}(t) \cdot \overline{\zeta}^\ast_{\widetilde{X},2}(t).
\]
\end{corollary}

\section{Examples}
Let  $f(x,y,z)$ be a weighted homogeneous polynomial with reduced system of weights $W=(w_1,w_2,w_3;d)$. The {\em Gorenstein parameter} $a_f$ of $f$ is defined to be
\[
a_f:=d-w_1-w_2-w_3.
\]

We now consider the classification of virtual singularities according to the Gorenstein parameter $a_{f}$ of the non-degenerate invertible polynomial $f$.

The classification of the non-degenerate invertible polynomials  $f$ with $[G_{f} : G_0]=2$ and with $a_{f}< 0$ can be extracted from \cite[Table~3]{ET}. From this we derive the classification of virtual singularities given in Table~\ref{TabGor<0}.
\begin{table}[h]
\begin{center}
\begin{tabular}{|c|c|c|c|c|c|}
\hline
Type & $p_1,p_2, p_3$ & $\hh$ & Name  & Dolgachev & Gabrielov \\
\hline
IIA & $2,3,6$ & $-y^2z+z^2+x^3+x^2y$ & $J_{1,-1}$ & $2,2;2,1$ & $2,3,2$ \\
IIB & $2,2,2k$ & $-x^2z+y^2+yz^k+xz^k$ & $A_{2k-1,-1}$ & $2,2k-1;1,1$ & $2,2,2k-2$ \\
${\rm IIB}^\sharp$ & $2,2,2k$ & $-x^2z+y^2+yz^k+xy$ & $A^\sharp_{2k-1,-1}$ & $k,1;k,1$ & $1,2,2k-1$ \\
\hline
\end{tabular}
\end{center}
\caption{Gorenstein parameter $<0$ cases} \label{TabGor<0}
\end{table}

One can also classify the non-degenerate invertible polynomials with $[G_{f} : G_0]=2$ with $a_{f}=0,1$. It turns out that there are no such polynomials with $a_{f}=0$. The virtual singularities corresponding to polynomials with $a_{f}=1$ are listed in 
Table~\ref{TabGor1}. 
\begin{table}[h] 
\begin{center}
\begin{tabular}{|c|c|c|c|c|c|}
\hline
Type & $p_1,p_2(q_2), p_3(q_3)$ & $\hh$ & Name  & Dolgachev & Gabrielov \\
\hline
IIA & 2,3,18 & $-y^4z+z^2+x^3+x^2y^3$ & $J_{3,-1}$ & 2,2;2,3 & 2,3,10 \\
IIB & 4,2,6 & $-x^3z+y^2+yz^3+x^2z^3$ & $Z_{1,-1}$ & 2,4;2,2 & 4,2,8  \\
${\rm IIB}^\sharp$ & 4,2,6 & $-x^3z+y^2+yz^3+x^2y$ & $W^\sharp_{1,-1}$ & 3,2;3,2 & 5,2,7  \\
IIB   & 6,2,4 & $-x^4z+y^2+yz^2+x^3z^2$ & $W_{1,-1}$ & 2,2;3,3 & 6,2,6 \\
${\rm IIB}^\sharp$  & 6,2,4 & $-x^4z+y^2+yz^2+x^3y$ & $W^\sharp_{1,-1}$ & 2,3;2,3 & 7,2,5 \\
III  &  2,2,4 & $-y^3z+z^2+x^3y+x^2y^3$ & $Z_{1,-1}$ & 2,2;2,4 & 2,4,8 \\
${\rm IV}_2$ &  3,6,18 & $-x^2z+xy^2+yz^3+xz^3$ & $S^\sharp_{1,-1}$ & 2,4;3,2 & 3,6,4 \\   
${\rm IV}_2^\sharp$ &  3,6,18 & $-x^2z+xy^2+yz^3+x^2y$ & $U_{1,-1}$ & 3,2;3,3 & 4,4,5 \\   
${\rm IV}_1$ &  3,12,24 & $-y^3z+x^3+yz^2+x^2y^2$ & $Q_{2,-1}$ & 2,2;2,5 & 3,3,7 \\
${\rm IV}_2$  & 5,10,20 & $-x^3z+xy^2+yz^2+x^2z^2$ & $S_{1,-1}$ & 2,2;4,3 & 5,5,3 \\
${\rm IV}_2^\sharp$ & 5,10,20 & $-x^3z+xy^2+yz^2+x^3y$ & $S^\sharp_{1,-1}$ & 2,3;2,4 & 6,3,4 \\
\hline
\end{tabular}
\end{center}
\caption{Gorenstein parameter 1 cases} \label{TabGor1}
\end{table}

I turns out that the virtual singularities with $a_{f}=1$  are exactly the virtual singularities corresponding to the bimodal series. 
According to Arnold's classification \cite{A}, there are 8 series of bimodal hypersurface singularities.
The virtual bimodal singularities are defined by setting $k=-1$ in the equations of these singularities. The names of Arnold are used in Table~\ref{TabGor1} and the equations for $k=-1$ are listed in Table~\ref{TabBi}. We compare them with our polynomials $\hh$. We also indicate the names of the dual isolated complete intersection singularities according to the notation of \cite{Wa}. It turns out that these are exactly the singularities in the extension of Arnold's strange duality of \cite{EW}.
\begin{table}[h]
\begin{center}
\begin{tabular}{|c|c|c|c|c|}
\hline
Series & Arnold's equation & Type  & $\hh(x,y,z)$ & Dual \\
\hline
$J_{3,-1}$ & $x^3+x^2y^3+z^2+y^8$ &  IIA & $x^3+x^2y^3+z^2-y^4z$ & $J'_9$\\
$Z_{1,-1}$ & $x^3y+x^2y^3+z^2+y^6$  & III  & $x^3y+x^2y^3+z^2-y^3z$ & $J'_{10}$\\
$Q_{2,-1}$ & $x^3+x^2y^2+yz^2+y^5$  &  ${\rm IV}_1$   & $x^3+x^2y^2+yz^2-y^3z$ & $J'_{11}$\\
$W_{1,-1}$ & $x^3z^2+y^2+z^4+x^5$ &  IIB  & $x^3z^2+y^2+yz^2-x^4z$ & $K'_{10}$\\
$W^\sharp_{1,-1}$ & $(x^3+z^2)^2+y^2+x^4z$  &  ${\rm IIB}^\sharp$ &  $x^3y+y^2+yz^2-x^4z$ & $L_{10}$\\
$S_{1,-1}$ & $xy^2+x^2z^2+yz^2+x^4$  & ${\rm IV}_2$  & $xy^2+x^2z^2+yz^2-x^3z$ & $K'_{11}$ \\
$S^\sharp_{1,-1}$ & $x^3y+xy^2+yz^2+x^3z$   &  ${\rm IV}_2^\sharp$ & $x^3y+xy^2+yz^2-x^3z$ & $L_{11}$ \\
$U_{1,-1}$ & $x^2y+y^3+yz^3+x^2z$ &  ${\rm IV}_2^\sharp$  & $x^2y+xy^2+yz^3-x^2z$ & $M_{11}$\\
\hline
\end{tabular}
\end{center}
\caption{Bimodal virtual singularities} \label{TabBi}
\end{table}

We indicate the values of the Dolgachev and Gabrielov numbers of the polynomials $\hh$ associated to the virtual bimodal singularites and the Dolgachev and Gabrielov numbers of the corresponding dual pairs of polynomials defining the isolated complete intersection singularities (ICIS)  in Table~\ref{TabBimon}.

\begin{table}
\begin{center}
\begin{tabular}{|c|c|c||c|c|c|}
\hline
Name & ${\rm Dol}(\hh)$ &  ${\rm Gab}(\hh)$  &  ${\rm Dol}(\widetilde{\ff}_1, \widetilde{\ff}_2)$ &  ${\rm Gab}(\widetilde{\ff}_1, \widetilde{\ff}_2)$ & Dual  \\
\hline
$J_{3,-1}$  & $2,2;2,3$ & $2,3,10$  & $2,3,10$ & $2,2;2,3$  & $J_9'$  \\
$Z_{1,-1}$ & $2,2;2,4$  & $2,4,8$ & $2,4,8$ & $2,2;2,4$ & $J_{10}'$  \\
$Q_{2,-1}$ & $2,2;2,5$  & $3,3,7$  & $3,3,7$ & $2,2;2,5$ & $J_{11}'$  \\
$W_{1,-1}$  & $2,2;3,3$ & $2,6,6$  & $2,6,6$ & $2,2;3,3$ & $K_{10}'$ \\
$W^\sharp_{1,-1}$  & $2,3;2,3$  & $2,5,7$ & $2,5,7$ & $2,3;2,3$ & $L_{10}$ \\
$S_{1,-1}$  & $2,2;3,4$  & $3,5,5$  & $3,5,5$ & $2,2;3,4$ & $K_{11}'$ \\
$S^\sharp_{1,-1}$  & $2,3;2,4$  & $3,4,6$  & $3,4,6$ & $2,3;2,4$ & $L_{11}$ \\
$U_{1,-1}$ & $2,3;3,3$ & $4,4,5$  & $4,4,5$ & $2,3;3,3$ & $M_{11}$  \\
\hline
\end{tabular}
\end{center}
\caption{Strange duality of virtual bimodal singularities and ICIS} \label{TabBimon}
\end{table}

Let $\hh(x,y,z)=0$ be the equation for one of the virtual bimodal singularities. By inspection, one sees that $\hh$ has besides the origin an additional critical point which is of type $A_1$. Now we want to consider Coxeter-Dynkin diagrams of these singularities. Let $X:=\{ (x,y,z) \in \CC^3 \, | \, \hh(x,y,z)=0\}$.
The function $\hh$ defines a locally trivial fibration $\hh : \CC^3 \setminus X \to \CC^\ast$. Let $V=\hh^{-1}(1) \cap X$ be the Milnor fibre of this fibration. We shall consider a (strongly ) distinguished basis of vanishing cycles of the homology group $H_2(V;\ZZ)$ (see e.g. \cite{E4}). 
The critical points outside the origin give additional vanishing cycles in $H_2(V;\ZZ)$. We define the {\em Milnor number} $\mu$ of $X$ to be the rank of $H_2(V;\ZZ)$. It is equal to the sum of the Milnor numbers of the singular points of $\hh$. It is indicated in Table~\ref{TabBiGab}.

In order to compute a Coxeter-Dynkin diagram for a distinguished basis of vanishing cycles we use the method of Gabrielov \cite{G}. We have to consider the polar curve corresponding to a choice of a linear function $z: \CC^n \to \CC$. The choice of the function is indicated in Table~\ref{TabBiGab}. The additional critical points lie on the polar curve. One can easily generalize the method of Gabrielov to include these additional critical points. By \cite{G}, one obtains an intersection matrix of a distinguished basis of $\hh$ from the one of a distinguished basis for $\hh|_{z=0}$ by the following formulas. Let $(e_j)$ ($j=1,2$ in case a), $j=1,2,3$ in case b) and $j=1,2,3,4$ in case c)) be a distinguished basis of $\hh|_{z=0}$ corresponding to
the Coxeter-Dynkin diagram presented in Fig.~\ref{Figfz=0}. Let $M_j$ be the numbers indicated in Table~\ref{TabBiGab}. Then there is a distinguished basis $(e_j^m, 1 \leq m \leq M_j)$ with the following intersection numbers
\begin{eqnarray*}
\langle e_j^m, e_{j'}^m \rangle & = & \langle e_j,e_{j'} \rangle, \\
\langle e_j^m, e_j^{m'} \rangle & = & 1 \quad \mbox{ for } |m'-m|=1, \\
\langle e_j^m, e_{j'}^{m'} \rangle & = & - \langle e_j,e_{j'} \rangle \quad \mbox{ for } |m'-m|=1 \mbox{ and } (m'-m)(j'-j)<0, \\
\langle e_j^m, e_{j'}^{m'} \rangle & = & 0 \quad \mbox{ for } |m'-m|>1 \mbox{ or } (m'-m)(j'-j)>0.
\end{eqnarray*}

\begin{figure}
$$
\xymatrix{ 
 & & &  & & & &    *{\bullet} \ar@{-}[dr] \ar@{}_{1}[l] & & &  \\
{\mbox{a)}} & *{\bullet} \ar@{-}[r] \ar@{}_{1}[d] & *{\bullet}  \ar@{}^{2}[d] & {\mbox{b)}} & *{\bullet} \ar@{-}[r] \ar@{}_{1}[d] & *{\bullet} \ar@{-}[r] \ar@{}_{3}[d] & *{\bullet} \ar@{}^{2}[d] & {\mbox{c)}} & *{\bullet} \ar@{-}[r] \ar@{}^{4}[d] &  *{\bullet} \ar@{-}[l] \ar@{}_{3}[r]  & \\
& & & & & & &    *{\bullet} \ar@{-}[ur] \ar@{}^{2}[l] & & & 
  } 
$$
\caption{Coxeter-Dynkin diagrams of a distinguished basis for $g|_{z=0}$} \label{Figfz=0}
\end{figure}
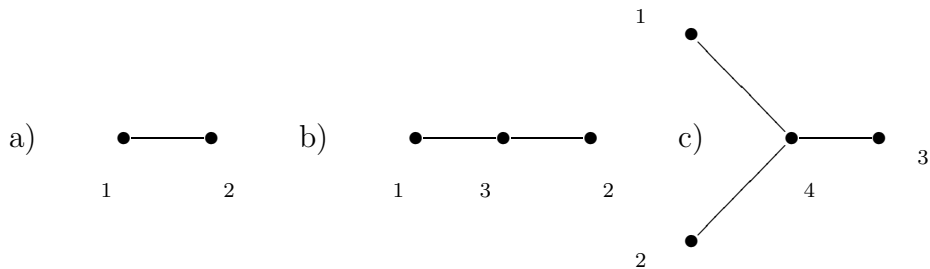

\begin{table}[h]
\begin{center}
\begin{tabular}{|c|c|c|c|c|}
\hline
Virtual  & Equation &  Numbers $M_j$ & $\gamma_1, \gamma_2, \gamma_3$ & $\mu$ \\
\hline
$J_{3,-1}$ & $x^2+y^3+y^2z^3-xz^4$ & $7+1,7$  & $2,3,9+1$ & 15\\
$Z_{1,-1}$ & $x^2+y^3(z-y)+y^2z^3-xz^3$  & $5+1,3,5$ & $2,4,7+1$ & 14\\
$Q_{2,-1}$ & $x^3+(z-x)y^2+x^2z^2-yz^3$ & $2,2,4+1,4$ & $3,3,6+1$ & 13\\
$W_{1,-1}$ & $x^2y+y^2+x^2z^3-xz^4$ & $5,4+1,4$ & $2,5+1,6$ & 14  \\
$W^\sharp_{1,-1}$ & $x^2y+y^2+yz^3-xz^4$  & $5+1,4,4$ & $2,5,6+1$ & 14\\
$S_{1,-1}$ & $x^2y+(z-y)y^2+x^2z^2-xz^3$ & $2,4,3+1,3$ & $3,5,4+1$ & 13\\
$S^\sharp_{1,-1}$ & $x^2y+(z-y)y^2+yz^3-xz^3$ & $2,4+1,3,3$ & $3,4,5+1$ & 13\\
$U_{1,-1}$ & $x^2y+xy^2+yz^3-x^2z$ & $3,3,3+1,3$ & $4,4,4+1$  & 13 \\
\hline
\end{tabular}
\end{center}
\caption{Coxeter-Dynkin diagrams of virtual bimodal singularities} \label{TabBiGab}
\end{table}

In Table~\ref{TabBiGab}, the contribution of the additional critical points to the numbers $M_j$ is indicated. By the sequences of elementary basis transformations indicated in \cite{E2}, the distinguished basis $(e_j^m)$ can be transformed to a distinguished basis 
\[ ( \delta_1; \delta_1^1, \delta_2^1, \ldots, \delta_{\gamma_1-1}^1; \delta_1^2, \delta_2^2, \ldots, \delta_{\gamma_2-1}^2; \delta_1^3, \delta_2^3, \ldots, \delta_{\gamma_3-1}^3;\delta_{\mu-1}, \delta_\mu ) \]
with a Coxeter-Dynkin diagram of the form of Fig.~\ref{FigSpqr} where $\gamma_1, \gamma_2, \gamma_3$ are the Gabrielov numbers of $X$. We call this graph $S_{\gamma_1,\gamma_2,\gamma_3}$.
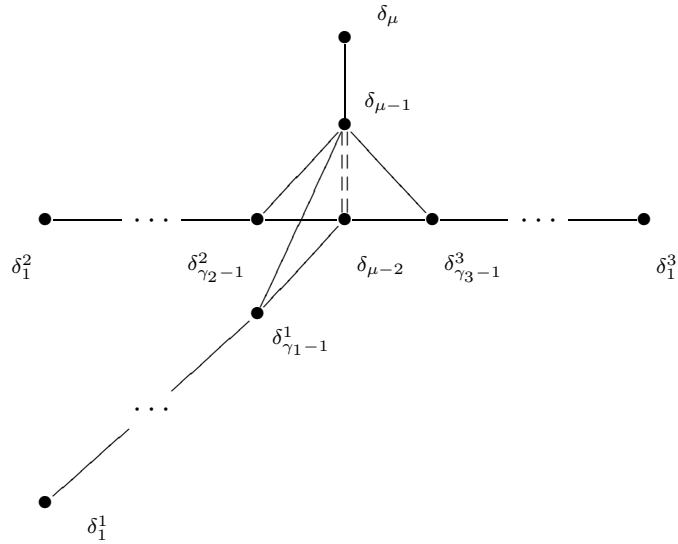
\begin{figure}
$$
\xymatrix{ 
 & & & *{\bullet} \ar@{-}[d] \ar@{}^{\delta_{\mu}}[r] & & & \\
 & & & *{\bullet} \ar@{==}[d] \ar@{-}[dr]  \ar@{-}[ldd] \ar@{}^{\delta_{\mu-1}}[r]
 & & &  \\
 *{\bullet} \ar@{-}[r] \ar@{}_{\delta^2_1}[d]  & {\cdots} \ar@{-}[r]  & *{\bullet} \ar@{-}[r] \ar@{-}[ur]   \ar@{}_{\delta^2_{\gamma_2-1}}[d] & *{\bullet} \ar@{-}[dl] \ar@{-}[r] \ar@{}^{\delta_{\mu-2}}[d] & *{\bullet} \ar@{-}[r]  \ar@{}^{\delta^3_{\gamma_3-1}}[d]  & {\cdots} \ar@{-}[r]  &*{\bullet} \ar@{}^{\delta^3_1}[d]   \\
 & &   *{\bullet} \ar@{-}[dl] \ar@{}_{\delta^1_{\gamma_1-1}}[r]  & & & & \\
 & {\cdots} \ar@{-}[dl] & & & & & \\
*{\bullet}  \ar@{}_{\delta^1_1}[r] & & & & & &
  } 
$$
\caption{The graph $S_{\gamma_1,\gamma_2,\gamma_3}$} \label{FigSpqr}
\end{figure}

Now let us consider the dual pair $(\widetilde{\ff}_1, \widetilde{\ff}_2)$ and the isolated complete intersection singularity defined by it. According to \cite[Proposition~3.6.1]{E1}, one can find a Coxeter-Dynkin diagram with respect to a distinguished basis of thimbles of the form $\Pi_{\gamma_1,\gamma_2,\gamma_3,\gamma_4}$ of Fig.~\ref{FigPipqrs} where $\gamma_1, \gamma_2; \gamma_3, \gamma_4$ are the Gabrielov numbers of the pair $(\widetilde{\ff}_1, \widetilde{\ff}_2)$.
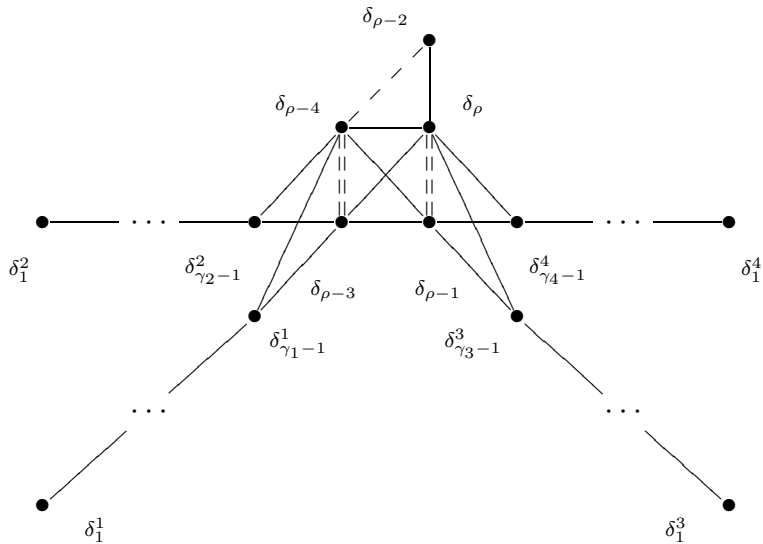
\begin{figure}
$$
\xymatrix{ 
 & & & & *{\bullet} \ar@{-}[d] \ar@{--}[dl] \ar@{}_{\delta_{\rho-2}}[l]  &  & & \\
 & & & *{\bullet} \ar@{-}[r] \ar@{==}[d] \ar@{-}[dr]  \ar@{-}[ldd] \ar@{}_{\delta_{\rho-4}}[l]
 & *{\bullet} \ar@{==}[d] \ar@{-}[dr]  \ar@{-}[rdd] \ar@{}^{\delta_{\rho}}[r]& & &  \\
 *{\bullet} \ar@{-}[r] \ar@{}_{\delta^2_1}[d]  & {\cdots} \ar@{-}[r]  & *{\bullet} \ar@{-}[r] \ar@{-}[ur]   \ar@{}_{\delta^2_{\gamma_2-1}}[d] & *{\bullet} \ar@{-}[dl] \ar@{-}[r] \ar@{-}[ur] \ar@{}^{\delta_{\rho-3}}[ld] & *{\bullet} \ar@{-}[r] \ar@{-}[dr] \ar@{}_{\delta_{\rho-1}}[rd] & *{\bullet} \ar@{-}[r]  \ar@{}^{\delta^4_{\gamma_4-1}}[d]  & {\cdots} \ar@{-}[r]  &*{\bullet} \ar@{}^{\delta^4_1}[d]   \\
& &  *{\bullet} \ar@{-}[dl] \ar@{}_{\delta^1_{\gamma_1-1}}[r]  & & & *{\bullet} \ar@{-}[dr] \ar@{}^{\delta^3_{\gamma_3-1}}[l] & &  \\
 & {\cdots} \ar@{-}[dl] & & & & & {\cdots} \ar@{-}[dr] & \\
*{\bullet}  \ar@{}_{\delta^1_1}[r] & & & & & & & *{\bullet}  
\ar@{}^{\delta^3_1}[l]
  } 
$$
\caption{The graph $\Pi_{\gamma_1,\gamma_2,\gamma_3,\gamma_4}$} \label{FigPipqrs}
\end{figure}


\end{document}